\documentclass[12pt,leqno,letterpaper]{article}

\usepackage{amsmath,amsthm,enumerate,amssymb}
\usepackage[latin1]{inputenc}
\usepackage[T1]{fontenc}
\usepackage[english]{babel}
\usepackage{graphicx}

\newtheorem{theorem}{Theorem}[section]

\newtheorem{lemma}[theorem]{Lemma}
\newtheorem{corollary}[theorem]{Corollary}
\newtheorem{definition}[theorem]{Definition}

\newcommand{\Tr}{{\rm Tr\hskip -0.2em}~}

\newcommand{\df}[2]{\frac{d#1}{d#2}}

\newcounter{std}

\begin{document}

\title{Non-commutative Hardy inequalities}
\author{Frank Hansen}
\date{June 4, 2008\\
{\tiny Revised version January 8, 2009}}

\maketitle

\begin{abstract}

We extend Hardy's inequality from sequences of non-negative numbers to sequences of positive semi-definite operators
if the parameter $ p $ satisfies $ 1<p\le 2, $ and to operators under a trace
for arbitrary $ p>1. $ Applications to trace functions are given. We introduce the tracial geometric mean
and generalize Carleman's inequality.\\[1ex]

\end{abstract}

\section{Introduction}

The standard form of Hardy's inequality \cite{kn:hardy:1920:1, kn:hardy:1925:1, kn:hardy:1967, kn:kufner:2006:1} asserts that
\begin{equation}\label{hardy's inequality}
\sum_{n=1}^\infty\left(\frac{1}{n}\sum_{k=1}^n a_k\right)^p
\le\left(\frac{p}{p-1}\right)^p \sum_{n=1}^\infty a_n^p
\end{equation}
for $ p>1 $ and any sequence $ a=(a_n) $ of non-negative real numbers in $ l_p. $ The inequality is sharp
in the sense that the constant $ (p/(p-1))^p $ cannot be replaced by a smaller number such that the inequality remains true for all (even finite) sequences of non-negative real numbers.

The arithmetic
averaging operator $ h $ (or the Hardy operator) is defined by setting
\begin{equation}\label{hardy's operator}
h(a)=\left(\frac{1}{n}\sum_{k=1}^n a_k\right)_{n=1}^\infty
\end{equation}
for any sequence $ a=(a_1,a_2,\dots) $ of real or complex numbers. Hardy's inequality thus states that
$ h $ maps $ l_p $ into itself and has norm $ p/(p-1). $

We shall prove (\ref{hardy's inequality}) also for sequences of positive semi-definite operators
on a Hilbert space provided $ 1<p\le 2. $ Under the trace the inequality is valid for arbitrary $ p>1 $
and (appropriate) sequences of positive semi-definite operators.
The classical proofs of Hardy's inequality are not easy to extend to operators.
The key to the results in this paper is the realization that, as is the case for many other classical inequalities, Hardy's inequality is in essence an expression of convexity.
The author is indebted to Lars-Erik Persson
who in a talk mentioned that there is a connection between convexity and Hardy's inequality,
and that it is part of the unpublished folklore known to specialist in classical inequalities.

\section{Non-commutative generalizations}

\begin{lemma}\label{convexity lemma}
Let $ 1\le p\le 2 $ be a real number, and let $ g\colon (0,\infty)\to B(H)_+ $ be a weakly measurable map such that the integral
\[
\int_0^\infty g(x)^p\, \frac{dx}{x}
\]
defines a bounded linear operator on $ H. $ Then the operator inequality
\[
\int_0^\infty\left(\frac{1}{x}\int_0^x g(t)\,dt\right)^p \frac{dx}{x}\le \int_0^\infty g(x)^p\, \frac{dx}{x}
\]
is valid. The inequality cannot (generally) be improved by inserting a constant less than one in front of the second integral.
\end{lemma}

\begin{proof}
Since the function $ t\to t^p $ is operator convex we obtain
\[
\left(\frac{1}{x}\int_0^x g(t)\,dt\right)^p\le \frac{1}{x}\int_0^x g(t)^p\,dt\qquad x>0
\]
with equality for functions $ g $ which are constant in $ (0,x]. $
Consequently
\[
\begin{array}{l}
\displaystyle\int_0^\infty\left(\frac{1}{x}\int_0^x g(t)\,dt\right)^p \frac{dx}{x}
\displaystyle\le\int_0^\infty\frac{1}{x}\int_0^x g(t)^p\,dt\,\frac{dx}{x}\\[3ex]
\displaystyle=\int_0^\infty g(t)^p\int_t^\infty \frac{1}{x^2}\,dx\,dt
=\int_0^\infty g(t)^p\,\frac{dt}{t}
\end{array}
\]
and the assertions are proved.
\end{proof}

\begin{lemma}\label{convexity under a trace lemma}
Let $ p\ge 1 $ be a real number, and let $ g\colon (0,\infty)\to B(H)_+ $ be a weakly measurable map such that the integral
\[
\int_0^\infty g(x)^p\, \frac{dx}{x}
\]
defines a bounded linear trace class operator on $ H. $ Then the inequality
\[
\Tr\int_0^\infty\left(\frac{1}{x}\int_0^x g(t)\,dt\right)^p \frac{dx}{x}\le \Tr\int_0^\infty g(x)^p\, \frac{dx}{x}
\]
is valid. The inequality cannot (generally) be improved by inserting a constant less than one in front of the second integral.
\end{lemma}

\begin{proof} The function $ t\to t^p $ is for $ p\ge 1 $ convex in operators under a trace. That is, the function
$ X\to\Tr X^p $ is convex in positive semi-definite operators $ X $ such that $ X^p $ is trace class. The same steps as in the proof of the
preceding lemma, mutatis mutandis, may be carried out and the assertions follow.
\end{proof}

\begin{theorem}\label{theorem: main}
Let $ 1<p\le 2 $ be a real number, and let $ f\colon (0,\infty)\to B(H)_+ $ be a weakly measurable map such that the integral
\[
\int_0^\infty \, f(x)^p dx
\]
defines a bounded linear operator on $ H. $ Then the inequality
\[
\int_0^\infty\left(\frac{1}{x}\int_0^x f(t)\,dt\right)^p dx
\le\left(\frac{p}{p-1}\right)^p \int_0^\infty \, f(x)^p dx
\]
holds, and the constant $ (p/(p-1))^p $ is the best possible.
\end{theorem}

\begin{proof}
We apply the function
\[
g(t)=f(t^{p/(p-1)}) t^{1/(p-1)}
\]
in Lemma~\ref{convexity lemma} and obtain
\[
\int_0^\infty\left(\frac{1}{x}\int_0^x f(t^{p/(p-1)}) t^{1/(p-1)}\,dt\right)^p \frac{dx}{x}
\le\int_0^\infty \, f(x^{p/(p-1)})^p x^{p/(p-1)}\frac{dx}{x}
\]
Setting $ y=t^{p/(p-1)} $ and thus
\[
\df{y}{t}=\frac{p}{p-1}\cdot t^{1/(p-1)}
\]
in the first integral, we obtain
\[
\begin{array}{l}
\displaystyle\int_0^\infty\left(\frac{1}{x}\int_0^x f(t^{p/(p-1)}) t^{1/(p-1)}\,dt\right)^p \frac{dx}{x}\\[3ex]
\displaystyle=\left(\frac{p-1}{p}\right)^p \int_0^\infty\left(\frac{1}{x}\int_0^{x^{p/(p-1)}} f(y)\,dy\right)^p \frac{dx}{x}\\[3ex]
\end{array}
\]
and thus
\[
\left(\frac{p-1}{p}\right)^p \int_0^\infty\left(\frac{1}{x}\int_0^{x^{p/(p-1)}} f(y)\,dy\right)^p \frac{dx}{x}
\le\int_0^\infty \, f(x^{p/(p-1)})^p x^{p/(p-1)}\frac{dx}{x}\,.
\]
We finally set $ z=x^{p/(p-1)} $ and thus
\[
\df{z}{x}=\frac{p}{p-1}\cdot \frac{z}{x}
\]
in both integrals and obtain
\[
\left(\frac{p-1}{p}\right)^{p+1} \int_0^\infty\left(\frac{1}{z^{(p-1)/p}}\int_0^z f(y)\,dy\right)^p \frac{dz}{z}
\le\frac{p-1}{p} \int_0^\infty \, f(z)^p dz.
\]
This is rearranged to
\[
\int_0^\infty\left(\frac{1}{z}\int_0^z f(y)\,dy\right)^p dz
\le\left(\frac{p}{p-1}\right)^p \int_0^\infty \, f(z)^p dz
\]
and the assertions are proved.
\end{proof}

\begin{theorem}
Let $ 1<p $ be a real number, and let $ f\colon (0,\infty)\to B(H)_+ $ be a weakly measurable map such that the integral
\[
\int_0^\infty \, f(x)^p dx
\]
defines a bounded linear trace class operator on $ H. $ Then the inequality
\[
\Tr\int_0^\infty\left(\frac{1}{x}\int_0^x f(t)\,dt\right)^p dx
\le\left(\frac{p}{p-1}\right)^p \Tr\int_0^\infty \, f(x)^p dx
\]
holds, and the constant $ (p/(p-1))^p $ is the best possible.
\end{theorem}

\begin{proof}
We apply for $ p>1 $ the function
\[
g(t)=f(t^{p/(p-1)}) t^{1/(p-1)}
\]
in Lemma~\ref{convexity under a trace lemma} and make the same substitutions under the trace as
applied in the proof of Theorem~\ref{theorem: main}.
\end{proof}

By considering the case where $ f $ is a step function we obtain:

\begin{corollary} Let $ 1<p\le 2 $ be a real number, and let $ a=(a_1,a_2,\dots) $ be a sequence of positive semi-definite operators on a Hilbert space such that the infinite sum $ a_1^p+a_2^p+\cdots $ defines a bounded operator. Then the inequality
\begin{equation}\label{hardy's operator inequality}
\sum_{n=1}^\infty\left(\frac{1}{n}\sum_{k=1}^n a_k\right)^p
\le\left(\frac{p}{p-1}\right)^p \sum_{n=1}^\infty a_n^p
\end{equation}
holds, and the constant $ (p/(p-1))^p $ is the best possible.
\end{corollary}

\begin{corollary} Let $ 1<p $ be a real number, and let $ a=(a_1,a_2,\dots) $ be a sequence of positive semi-definite operators on a Hilbert space such that the infinite sum $ a_1^p+a_2^p+\cdots $ defines a bounded trace class operator.
Then the inequality
\begin{equation}\label{hardy's tracial operator inequality}
\Tr\sum_{n=1}^\infty\left(\frac{1}{n}\sum_{k=1}^n a_k\right)^p
\le\left(\frac{p}{p-1}\right)^p \Tr\sum_{n=1}^\infty a_n^p
\end{equation}
holds, and the constant $ (p/(p-1))^p $ is the best possible.
\end{corollary}

\subsection{Trace functions}

Consider for $ p> 0 $ the homogeneous functional
\[
\Phi_p(a)=\Tr\left[\left(\sum_{n=1}^\infty a_n^p\right)^{1/p}\right]
\]
defined on sequences $ a=(a_1,a_2,\dots) $  of positive definite operators acting on a Hilbert space such that
$ (a_1^p+a_2^p+\cdots)^{1/p} $ defines a bounded trace class operator. It is a consequence of 
Epstein's theorem that $ \Phi_p $
is concave,  cf. \cite{kn:epstein:1973} and \cite[Theorem 1.1]{kn:carlen:2008}. Carlen and Lieb conjectured in 1999 that $ \Phi_p $ is convex for $ 1\le p\le 2, $ and they showed that this would lead to a Minkowski inequality for operators on a tensor product of three Hilbert spaces and to another proof of strong subadditivity of entropy. Carlen and Lieb eventually proved the conjecture in 2008 \cite[Theorem 1.1]{kn:carlen:2008}.

Since $ \Phi_p $ is convex for $ 1\le p\le 2 $ the domain is a convex cone, but the functional cannot be extended to a norm. The obstacle is that the modulus $ |X| $ of an operator $ X $ is not a convex function in $ X. $

\begin{theorem} Let $ 1<p\le 2 $ be a real number, and let $ a=(a_1,a_2,\dots) $ be a sequence of positive semi-definite operators on a Hilbert space such that
$ (a_1^p+a_2^p+\cdots)^{1/p} $ defines a bounded trace class operator. Then
\[
\Phi_p(h(a))\le\frac{p}{p-1}\,\Phi_p(a),
\]
where $ h $ is the Hardy operator (\ref{hardy's operator}).
\end{theorem}

\begin{proof}
Since $ t\to t^\alpha $ is operator monotone for $ 0\le\alpha\le 1, $ it follows from the non-commutative Hardy inequality
(\ref{hardy's operator inequality}) that
\[
\left(\sum_{n=1}^\infty\left[\frac{1}{n}\sum_{k=1}^n a_k\right]^{p\:}\right)^{1/p}
\le\frac{p}{p-1}\left(\sum_{n=1}^\infty a_n^p\right)^{1/p}
\]
for $ 1<p\le 2. $ The assertion then follows by taking the trace.
\end{proof}

\section{The tracial geometric mean}

The operator power mean\footnote{Sometimes $ p $ is replaced with $ 1/p $ in the literature
\cite{kn:furuta:2005}.} is defined by
\[
M_p(a_1,\dots,a_n)=\left(\frac{1}{n}\sum_{k=1}^n a_k^{1/p}\right)^p.
\]
Let now $ q\ge p\ge 1. $ Since $ t\to t^{p/q} $ is operator concave we obtain
\[
M_p(a_1,\dots,a_n)^{1/q}=\left(\frac{1}{n}\sum_{k=1}^n a_k^{1/p}\right)^{p/q}
\ge \frac{1}{n}\sum_{k=1}^n a_k^{1/q}=M_q(a_1,\dots,a_n)^{1/q},
\]
and since $ t\to t^q $ is monotone under a trace, that is $ X\to \Tr X^q $ is monotone increasing in positive definite $ X, $ we also obtain
\[
\Tr M_p(a_1,\dots,a_n)\ge\Tr M_q(a_1,\dots,a_n).
\]
The trace of the operator mean is thus a decreasing function in the power parameter $ p, $
so the following definition makes sense:

\begin{definition}
We define the tracial geometric mean by setting
\[
\text{TG}(a_1,\dots,a_n)=\lim_{p\to\infty}\Tr \left(\frac{1}{n}\sum_{k=1}^n a_k^{1/p}\right)^p
\]
for positive semi-definite trace class operators $ a_1,\dots,a_n $ on a Hilbert space.
\end{definition}

Notice (by setting $ p=1) $ the inequality
\[
\Tr\left(\frac{a_1 +\cdots+a_n}{n}\right) \ge \text{TG}(a_1,\dots,a_n).
\]
The tracial geometric mean has a number of attractive properties:

\begin{theorem} Let $ a_1,\dots,a_n $ be positive semi-definite bounded operators on a Hilbert space
and consider the tracial geometric mean $ TG(a_1,\dots,a_n). $

 \begin{list}{(\arabic{std})}{\usecounter{std}}

\item It coincides with the geometric mean $ (a_1\cdots a_n)^{1/n} $ for positive numbers and with $ \Tr (a_1\cdots a_n)^{1/n} $ for commuting operators.

\item It is a (tracial) mean; $ \text{TG}(a,\dots,a)=\Tr a. $

\item It is homogeneous; $ \text{TG}(t a_1,\dots,t a_n)=t\cdot\text{TG}(a_1,\dots,a_n) $ for $ t>0. $

\item It is symmetric in the entries.

\item It is unitarily invariant.

\item It acts additively on sequences of block matrices.

\item It is (separately) increasing in the entries.

\item It is (jointly) concave in the entries.

\end{list}

\end{theorem}

\begin{proof}

The first statement for positive numbers is a classical result and $ (2) $ to $ (6) $ are quite obvious.

To prove $ (7) $ we notice that $ M_p(a_1,\dots,a_n)^{1/p} $ is separately increasing
in the entries by operator monotonicity of $ t\to t^{1/p}. $ We therefore obtain, as above, that
$ \Tr M_p(a_1,\dots,a_n), $ for each $ p\ge 1, $ is separately increasing in the entries. This property
is hence satisfied also for the tracial geometric mean by taking the limit $ p\to\infty. $

To prove $ (8) $ we notice that $ \Tr M_p(a_1,\dots,a_n), $ for each $ p\ge 1, $ is a concave function in the entries by Epstein's theorem, cf. \cite{kn:epstein:1973} and \cite[Theorem 1.1]{kn:carlen:2008}. Taking the limit $ p\to\infty $ we therefore realize that also the tracial geometric mean is a concave function in the entries.
\end{proof}

The last property $ (8) $ is by far the most difficult to establish. The concrete formula given in the next theorem is related to \cite[Theorem 4.19]{kn:furuta:2005}. It is interesting to notice that it is of no use to prove $ (8) $ in the previous result.

\begin{theorem}
If the operators $ a_1,\dots,a_n $ are strictly positive, then the tracial geometric mean
\begin{equation}\label{formula: tracial geometric mean}
TG(a_1,\dots,a_n)=\Tr\exp\left(\frac{1}{n}\sum_{k=1}^n \log a_k\right).
\end{equation}
\end{theorem}

\begin{proof}
Since the logarithm is operator concave, we have
\[
p\cdot\log\left(\frac{1}{n}\sum_{k=1}^n a_k^{1/p}\right)\ge\frac{1}{n}\sum_{k=1}^n \log a_k
\]
and hence under the trace
\[
\begin{array}{rl}
M_p(a_1,\dots,a_k)&\displaystyle
=\Tr\exp\left[p\cdot\log\left(\frac{1}{n}\sum_{k=1}^n a_k^{1/p}\right)\right]\\[4ex]
&\displaystyle\ge\Tr\exp\left(\frac{1}{n}\sum_{k=1}^n \log a_k\right)
\end{array}
\]
for any $ p\ge 1. $ Taking the limit $ p\to\infty $ we thus obtain that the left hand side in (\ref{formula: tracial geometric mean}) majorizes the right hand side.
To verify the opposite inequality we use $ \log t\le t-1 $ to obtain
\[
\begin{array}{rl}
\log M_p(a_1,\dots,a_k)&\displaystyle
=p\cdot\log\left(\frac{1}{n}\sum_{k=1}^n a_k^{1/p}\right)\\[4ex]
&\displaystyle\le\frac{1}{n}\sum_{k=1}^n p(a_k^{1/p} - 1)
\end{array}
\]
and hence under the trace
\[
TG(a_1,\dots,a_n)\le
\Tr M_p(a_1,\dots,a_n)\le\Tr\exp\left(\frac{1}{n}\sum_{k=1}^n p(a_k^{1/p} - 1)\right)
\]
for any $ p\ge 1. $ Since $ p(a_k^{1/p}-1)\to \log a_k $ in the strong operator topology for $ p\to\infty, $ the assertion follows.
\end{proof}

The following theorem is a non-commutative version of Carleman's inequality \cite{kn:carleman:1923}.

\begin{theorem}
The inequality
\[
\sum_{n=1}^\infty \text{TG}(a_1,\dots,a_n)\le e\sum_{n=1}^n \Tr a_n
\]
is valid for sequences of positive semi-definite trace class operators such that the sum on the right hand side is finite.
\end{theorem}

\begin{proof}
If we in the tracial Hardy inequality (\ref{hardy's tracial operator inequality})
replace $ a_n $ with $ a_n^{1/p} $ for each $ n $ then we obtain
\[
\Tr\sum_{n=1}^\infty\left(\frac{1}{n}\sum_{k=1}^n a_k^{1/p}\right)^p
\le\left(\frac{p}{p-1}\right)^p \Tr\sum_{n=1}^\infty a_n
\]
for $ p>1. $ Since
\[
\left(\frac{p}{p-1}\right)^p\to e\quad\text{for}\quad p\to\infty
\]
the statement follows by letting $ p\to\infty $ in the inequality.
\end{proof}

\nocite{kn:kufner:2003:1}
\nocite{kn:kufner:2006:1}
\nocite{kn:ando:2004:1}
\nocite{kn:bhatia:2006:1}

{\small

% \bibliographystyle{plain}
%    \bibliography{mathharv}

\begin{thebibliography}{10}

\bibitem{kn:ando:2004:1}
T.~Ando, C.-K. Li, and R.~Mathias.
\newblock Geometric means.
\newblock {\em Linear Algebra Appl.}, 385:305--334, 2004.

\bibitem{kn:bhatia:2006:1}
R.~Bhatia and J.~Holbrook.
\newblock Noncommutative geometric means.
\newblock {\em Math. Intelligencer}, 28:32--39, 2006.

\bibitem{kn:carleman:1923}
T.~Carleman.
\newblock Sur les fonctions quasi-analytiques.
\newblock Fifth Scandinavian Congress of Mathematicians (Helsinki, 1922), pages
  181--197. Akadem. Buchh., 1923.

\bibitem{kn:carlen:2008}
E.A. Carlen and E.H. Lieb.
\newblock A \uppercase{M}inkowsky type trace inequality and strong
  subadditivity of quantum entropy \uppercase{II}: \uppercase{C}onvexity and
  concavity.
\newblock {\em Lett. Math. Phys.}, 83:107--126, 2008.

\bibitem{kn:epstein:1973}
H.~Epstein.
\newblock Remarks on two theorems of \uppercase{E}. \uppercase{L}ieb.
\newblock {\em Comm. Math. Phys.}, 31:317--325, 1973.

\bibitem{kn:furuta:2005}
T.~Furuta, Hot J.M., J.~Pe\v{c}ari\'{c}, and Y.~Seo.
\newblock {\em Mond-Pe\v{c}ari\'{c} Method in Operator Inequalities}.
\newblock Element, Zagreb, 2005.

\bibitem{kn:hardy:1920:1}
G.H. Hardy.
\newblock Note on a theorem of \uppercase{H}ilbert.
\newblock {\em Math. Z.}, 6:314--317, 1920.

\bibitem{kn:hardy:1925:1}
G.H. Hardy.
\newblock Notes on some points in the integral calculus, \uppercase{LX}.
  \uppercase{A}n inequality between integrals.
\newblock {\em Messenger of Math.}, 54:150--156, 1925.

\bibitem{kn:hardy:1967}
G.H. Hardy, J.E. Littlewood, and G.~P\'{o}lya.
\newblock {\em Inequalities, 2nd ed.}
\newblock {Cambridge University Press}, Cambridge, 1967.

\bibitem{kn:kufner:2006:1}
A.~Kufner, L.~Maligranda, and L.-E. Persson.
\newblock The prehistory of the \uppercase{H}ardy inequality.
\newblock {\em American Math. Month.}, pages 715--732, October 2006.

\bibitem{kn:kufner:2003:1}
A.~Kufner and L.-E. Persson.
\newblock {\em Weighted inequalities of \uppercase{H}ardy type}.
\newblock {World Scientific}, Singapore, 2003.

\end{thebibliography}

\vfill

      \noindent Frank Hansen: Department of Economics, University
       of Copenhagen, Studiestraede 6, DK-1455 Copenhagen K, Denmark.}

\end{document}